\documentclass[a4paper,11pt]{amsart}

\usepackage{amssymb}
\usepackage{amsmath,amsthm,amscd,enumerate, verbatim}

\usepackage{amscd}
\usepackage{xspace}
\usepackage{longtable}
\usepackage[all]{xy}
\usepackage{hyperref}

\DeclareMathOperator{\codim}{codim}

\newtheorem{proposition}{Proposition}
\newtheorem{theorem}{Theorem}
\newtheorem{corollary}{Corollary}
\newtheorem{lemma}{Lemma}
\theoremstyle{definition}

\newtheorem{remark}{\textbf{Remark}}
\newtheorem{notation}{\textbf{Construction}}

\newcommand{\mgn}{\mathcal{M}_{g,n} }
\newcommand{\virg}[1]{\textquotedblleft#1\textquotedblright}

\begin{document}
\title{{Harer Stability and Orbifold Cohomology}}
\author{Nicola Pagani}
\address{Institut f\"ur Algebraische Geometrie, Leibniz Universit\"at Hannover,
and Department of Mathematical Sciences, The University of Liverpool.}
\email{npagani@math.uni-hannover.de, pagani@liv.ac.uk}

\begin{abstract} In this paper we review the combinatorics of the twisted sectors of $\mgn$, and we exhibit a formula for the age of each of them in terms of the combinatorial data. Then we show that orbifold cohomology of $\mgn$ when $g \to \infty$ reduces to its ordinary cohomology. We do this by showing that the twisted sector of minimum age is always the hyperelliptic twisted sector with all markings in the Weierstrass points; the age of the latter moduli space is just half its codimension in $\mgn$. 
\end{abstract}

\maketitle
\section{introduction} In recent years there have been lots of new results on geometrical and topological properties of the moduli space $\mgn$  parametrizing smooth curves of genus $g$ with $n$ distinct marked points on it. When $2g-2+n>0$, this moduli space is a smooth Deligne--Mumford stack, or an orbifold, and its coarse moduli space is a quasiprojective variety of dimension $3g-3+n$. When $n>2g+2$, every marked curve is rigid, therefore the moduli space is actually a smooth quasiprojective variety.

A celebrated result states that there are isomorphisms
\begin{equation} \label{harerstability}
H^k (\mgn, \mathbb{Q}) \cong H^k(\mathcal{M}_{g+1,n}, \mathbb{Q})  \quad \textrm{when } 3k+2 \leq 2g.
\end{equation}
These isomorphisms are originarily due to Harer in \cite{harer}, but the ranges of their validity have been gradually improved over time by the efforts of different authors. This allows the definition of the \emph{stable cohomology}, denoted $H^*(\mathcal{M}_{\infty,n}, \mathbb{Q})$. The tautological classes $\kappa$ and $\psi$ are preserved by the above isomorphisms when $g$ is sufficiently large. A recent result, whose proof was completed by Madsen-Weiss in \cite{madsen}, asserts that the resulting maps
\begin{equation}\label{mw}
\mathbb{Q}[\kappa_1, \kappa_2, \ldots] \otimes \mathbb{Q}[\psi_1, \ldots, \psi_n] \to H^*(\mathcal{M}_{\infty,n}, \mathbb{Q})
\end{equation}
are also isomorphisms. (The statement that the homomorphisms \eqref{mw} are actually isomorphisms incorporates the theorem of Madsen and Weiss, which deals with the $n=0$ case, and its extension to the $n>0$ case, which follows then from Looijenga's result \cite[Proposition 2.1]{looijenga}). We refer the reader to \cite{kirwan} and \cite{wahl} for a survey of these topological results.

In the latest years, building on earlier results in topology \cite{kawasaki} and theoretical physics \cite{dhvw1}, \cite{dhvw2}, it has become clearer that when studying the geometry and topology of orbifolds, one should include in the study the \emph{twisted sectors} of the orbifold itself. We refer to \cite{alr} for an introduction to this emerging new subject. In particular, the cohomology theory of an orbifold is enriched by the so-called \emph{orbifold cohomology}, introduced by Chen and Ruan in \cite{chenruan}. As a graded vector space, the orbifold cohomology is the direct sum of the cohomology of the original orbifold and of the cohomology of the twisted sectors; the degree of the cohomology classes of each twisted sector is shifted in orbifold cohomology by (twice) a rational number called \emph{age}. This number is not of topological nature, in fact it depends on the complex structure. Its geometric significance appears in \cite{jkk} as the virtual rank of an element in the rational $K$-theory of the twisted sector also known as ``half of the normal bundle'', this element plays a key role in orbifold intersection theory. 

In this note, we introduce the twisted sectors of $\mgn$ in the combinatorial description of \cite{pagani2}, \cite{pt}, we write a closed formula for the age of the twisted sectors of $\mgn$ (special cases of this formula are in \cite{pagani2} for $\mathcal{M}_{2,n}$ and in \cite{pt} for $\mathcal{M}_g$). Our main result is Theorem \ref{teoremaeta}, which states that for fixed $(g,n)$, the twisted sector of minimum age is the hyperelliptic twisted sector with marked Weierstrass points. It is a well-known and classical fact, which we review in Proposition \ref{codimensione}, that the twisted sectors of $\mathcal{M}_{g,n}$ have codimension higher than $g-2+n$, with equality only for the hyperelliptic locus. Our novel contribution here is that the virtual rank of ``half of the normal bundle'' (see above) is strictly greater than $\frac{g-2+n}{2}$, with equality only for the hyperelliptic twisted sector. This inequality might have further geometric consequences, besides the implications in orbifold cohomology investigated in this note. (The study of the age of the twisted sectors of various types of moduli spaces of curves, has also recently played a significant role in the investigation of the singularities of the coarse moduli space.)

Combining Theorem \ref{teoremaeta} with Harer stability, we obtain that the orbifold cohomology of $\mgn$ stabilizes. Combining further our main result with the theorem of Madsen--Weiss, we can explicitly compute the orbifold cohomology of $\mgn$ in low degrees. Indeed, from Theorem \ref{teoremaeta}, we deduce
\begin{equation} \label{stablecoom}
H^k_{orb}(\mgn, \mathbb{Q})= H^k(\mgn, \mathbb{Q}), \quad \textrm{if } k < g-2+n, \textrm{ or } n>2g+2.
\end{equation}
(There are no twisted sectors of $\mgn$ if and only if $n>2g+2$).

The stabilization of orbifold cohomology was conjectured by Fantechi in the discussion following her talk \cite{fantechi} at \texttt{MSRI}. We acknowledge her for the insight in this topic. We also thank Stefano Maggiolo for having significantly improved the computer program that plays a role at the end of the proof of our main result. The author was supported by DFG project Hu 337/6-2. 
%

\section{The twisted sectors of $\mathcal{M}_{g,n}$ and their age}
In this section we review the combinatorics of the twisted sectors of $\mgn$. This description of the twisted sectors of $\mgn$ was obtained in \cite{pagani2} for $n\geq 1$ or $g=2$, and in \cite{pt} for the remaining cases $\mathcal{M}_{g,0}$, $g \geq 3$.

Let us fix $(g,n)$ with $2g-2+n>0$. A $(g,n)$-admissible datum consists of non-negative integers $(g', N; d_1, \ldots, d_{N-1}, a_1, \ldots, a_{N-1})$, with $N \geq 2$ and such that

\begin{equation} \label{rh}
2g-2=N(2g'-2)+ \sum_{i=1}^{N-1} (N- \gcd(i,N)) d_i,
\end{equation}
\begin{equation} \label{fan}
\sum_{i=1}^{N-1} i \ d_i \equiv 0 \pmod N,
\end{equation}
\begin{equation} \label{markedpoints}
\sum_{i=1}^{N-1} a_i=n, \quad a_i \leq d_i, \quad a_i=0 \textrm{ if } \gcd(i,N) \neq 1. 
\end{equation}
\begin{equation} \label{nonempty}
n=g'=0 \implies \textrm{the g.c.d. of } N \textrm{ and of the } i\textrm{'s such that } d_i \neq 0 \textrm{ is } 1.
\end{equation}

Each $(g,n)$-admissible datum corresponds to $\binom{n}{a_1, \ldots, a_{N-1}}$ twisted sectors of $\mgn$ that are related each to the other by an $(a_1, \ldots, a_{N-1})$-permutation of the $n$ marked points. Since we will only investigate properties of the twisted sectors of $\mgn$ that do not depend on this permutation, from now on we shall slightly abuse the notation and identify each twisted sector $Y$ of $\mgn$ with its $(g,n)$-admissible datum
\begin{displaymath}
Y \sim (g', N; d_1, \ldots, d_{N-1}, a_1, \ldots, a_{N-1}).
\end{displaymath} 
These facts follow from  \cite[Proposition 2.13]{pagani2} for $n \geq 1$ and from \cite[Corollary 2.16, Theorem 2.19]{pt} in the case $n=0$. 

We observe that, from condition \eqref{rh}, there are no $(g,n)$-admissible data when $n> 2g+2$; in particular, this is the case when $g$ equals $0$.

For completeness, we briefly recall our description of the twisted sectors of $\mgn$, from which the above correspondence follows. For more details, we refer to  \cite[Section 2.b]{pagani2} for the case $n \geq 1$ and to \cite[Section 2.b]{pt} for the case $n=0$. 
\begin{notation} 
A twisted sector of $\mgn$ parametrizes connected cyclic covers of order $N$ of curves of genus $g'$ with total space a curve of genus $g$, where the $n$ marked points are chosen among the points of total ramification. The branch divisor of the cyclic cover splits into $N-1$ divisors, some of which are possibly empty. Indeed to any point $p$ in the branch divisor $D$, let $q$ be any point in the fiber of $p$ under the cyclic cover map; we define $H_p$ as the stabilizer of the action of $\mathbb{Z}/N\mathbb{Z}$ at $q$, and $\psi_p$ as the character of the action of $H_p$ on the cotangent space in $q$. Then, for $0 <i<N$, we define $D_i$ as the subset of $D$ of those $p$ such that $H_p$ equals the subgroup generated by $i$ in $\mathbb{Z}/N\mathbb{Z}$, and such that $\psi_p(i)$ equals $\omega_N$, a fixed generator for $\mu_N$: the group of $N$-th roots of $1$. In addition to $(g,g',N)$, the admissible datum consists of $d_i:=|D_i|$ and of $a_i$, the number of chosen marked points in the preimage of $D_i$ under the cyclic cover map. 

Given a $(g,n)$-admissible datum, we can construct a moduli space of cyclic covers as in the paragraph above. Condition \eqref{rh} is Riemann--Hurwitz formula, condition \eqref{fan} is a compatibility condition that guarantees the existence of a (not necessarily connected) cyclic cover with the data $d_i$ and $N$, condition \eqref{markedpoints} corresponds to the fact that the marked points must be points of total ramification for the cover.  Now if $n\geq 1$, it is easy to see that the total space of the cover is forced to be connected and that the moduli space parametrizing such covers is also connected. If instead $n=0$, it is shown in \cite[Theorem 2.19]{pt} that there is always one connected component of the moduli space that parametrizes \emph{connected} cyclic covers. This component may possibly be empty only when $g'=0$, condition \eqref{nonempty} rules out precisely these cases. \label{constr1}
\end{notation}


Let us fix a twisted sector $(g', N; d_1, \ldots, d_{N-1}, a_1, \ldots, a_{N-1})$. Since $Y$ admits a finite map to $\mathcal{M}_{g', \sum d_i}$, its dimension is $3g'-3+ \sum d_i$ and its \emph{codimension in} $\mathcal{M}_{g,n}$ is
\begin{equation} \label{codim}
\codim(Y):= 3g-3g'- \sum_{i=1}^{N-1} d_i+n,
\end{equation}
its \emph{twin} is  $(g', N; d_{N-1}, \ldots, d_1, a_{N-1}, \ldots, a_1)$.
If $(g,n)$ is fixed and $n\leq2g+2$, the \emph{hyperelliptic twisted sector with $n$ marked Weierstrass points} is $(g'=0,N=2;d_1=2g+2,a_1=n)$. In short, we will also call it simply the \emph{hyperelliptic twisted sector}, from \eqref{codim} it has codimension $g-2+n$. We now review the following well-known fact.
\begin{proposition} \label{codimensione} The codimension of any twisted sector $Y$ of $\mathcal{M}_{g,n}$ satisfies $\codim(Y)\geq g-2+n$, with equality if and only if $Y$ is the hyperelliptic twisted sector with $n$ marked Weierstrass points.
\end{proposition}

\begin{proof}
Using formula \eqref{codim}, our statement is reduced to proving the inequality:
\begin{equation}
\sum d_i \leq 2g - 3 g' +2.
\end{equation}
Using Formula \eqref{rh}, we have that:
\begin{equation}\label{servedopo}
\frac{N}{2} \sum d_i \leq  \sum d_i (N-\gcd(i,N))= 2g-2 -N (2g'-2),
\end{equation}
therefore, it is enough to show that:
\begin{displaymath}
\frac{2}{N} (2g-2-N (2g'-2)) \leq 2 g - 3 g' + 2.
\end{displaymath}
Or, rearranging the terms, that:
\begin{equation}
(2 N -4) (g-1) + N g' \geq 0.
\end{equation}
This is clearly always true. Equality holds if and only if $g'$ equals $0$ and $N$ equals $2$. 
\end{proof}

Every twisted sector $Y$ is assigned a rational number, first defined by Chen--Ruan in \cite{chenruan}, which is called \emph{degree shifting number}, \emph{age}, or \emph{fermionic shift}. Orbifold cohomology is then the direct sum of the ordinary cohomology and of the cohomology of all the twisted sectors, where the latter is shifted in degree by twice the age. For completeness, we briefly review the Chen--Ruan definition of degree shifting number, building on Construction \ref{constr1}. 

\begin{notation} \label{defage} Let $f \colon Y \to \mathcal{M}_{g,n}$ be the natural map from the twisted sector to the moduli stack of curves. 
The group $\mu_N$ of $N$th roots of $1$ acts on $f^*(T_{\mathcal{M}_{g,n}})$, the action can be diagonalized, and each eigenvalue at a point of $Y$ has the form $\lambda_k= e^{2 \pi i \alpha_k}$, where the $\alpha_k \in [0, 1[ \cap \mathbb{Q}$ are the \virg{logarithms of the eigenvalues}. It is not difficult to see that the function $\sum_k \alpha_k$ is (well defined) and constant on $Y$, thus the \emph{age} of $Y$ is defined as
\begin{equation} \label{age} 
a(Y):= \sum_k \alpha_k \in \mathbb{Q}.
\end{equation}
 Moreover, by the very definition of twisted sector, the action of $\mu_N$ on $T_Y$ is trivial, thus in the definition \eqref{age} it is equivalent to sum the \virg{logarithms of the eigenvalues} of the normal bundle $N_Y \mathcal{M}_{g,n}$, where the latter is defined by the exact sequence of vector bundles 
 \begin{displaymath}
0 \to T_Y \to f^*(T_{\mathcal{M}_{g,n}}) \to N_Y \mathcal{M}_{g,n} \to 0.
\end{displaymath}
The age of a twisted sector can be interpreted as the virtual rank of an element in the rational $K$-theory of $Y$ that plays an important role in orbifold intersection theory, see \cite[Definition 1.3, Sections 1.3 and 4]{jkk}.
 \end{notation}

The age of a twisted sector of $\mgn$ can explicitly be determined in terms of its admissible datum. From \cite[Proposition 5.6]{pt} and \cite[Lemma 4.6]{pagani2}, we have the following formula for the age:
\begin{align}\label{etamg}
 a(Y) =& \frac{(3 g'-3)(N-1)}{2}+ \frac{1}{N} \sum_{\gcd(i,N)=1} a_i \sum_{k=1}^{N-1} k \sigma(k,i)+   \\ \nonumber &+ \frac{1}{N} \sum_{i=1}^{N-1} d_i \sum_{k=1}^{N-1} k \left( \left\{\frac{k i}{N}\right\} + \sigma(k,i) \right) ,
\end{align}
where $\{x\}:= x-\lfloor x \rfloor$ denotes the fractional part of $x\in \mathbb{Q}^+$, and \begin{displaymath}\sigma(k,i):=\begin{cases}0 &  ki+\gcd(i,N)\equiv 0 \pmod N \\ 1 & k i + \gcd(i,N) \not\equiv 0 \pmod N.\end{cases} \end{displaymath}

Using only \eqref{etamg}, it is an easy exercise to check that, if $Y$ and $Y'$ are twins, the following holds:
\begin{equation} \label{etacodim}
a(Y)+ a(Y')= \codim(Y)= \codim(Y').
\end{equation}
For example, when a twisted sector $Y$ is twin to itself (this happens always, for example, when $N=2$), its age is half its codimension.

\section{The twisted sectors of minimum age} 

Using only the combinatorial description of the previous section, and in analogy with Proposition \ref{codimensione}, we can prove the main result of this note. From now on, we assume $2g-2+n>0$.
\begin{theorem} \label{teoremaeta} The age of any twisted sector $Y$ satisfies $2 a(Y)\geq g-2+n$, with equality if and only if $Y$ is the hyperelliptic twisted sector with $n$ marked Weierstrass points. 
\end{theorem}
\noindent The marked hyperelliptic twisted sector is, using the terminology established in the previous section, twin to itself. Therefore its age is half its codimension: $\frac{g-2+n}{2}$.

From this, the following corollary relevant for orbifold cohomology follows:
\begin{equation} \label{stabilization}
H^k(\mathcal{M}_{g,n}, \mathbb{Q})= H^k_{orb}(\mathcal{M}_{g,n}, \mathbb{Q}) \quad\textrm{if } k < g-2+n, \textrm{ or } n> 2g+2.
\end{equation}
There are no twisted sectors of $\mathcal{M}_{g,n}$ if and only if $n>2g+2$; otherwise our bound on the cohomological degree $k$ is sharp.

Using the stability results for ordinary cohomology, we deduce:
\begin{corollary} \label{corollario} The isomorphisms \eqref{harerstability} are, in fact, isomorphisms
\begin{displaymath}H^k_{orb} (\mgn, \mathbb{Q}) \cong H^k_{orb}(\mathcal{M}_{g+1,n}, \mathbb{Q}),  \quad \textrm{when } k \leq \min (g-3+n, 2g/3-2/3).
\end{displaymath}

\end{corollary}
 In particular, we can interpret this by saying that orbifold cohomology of $\mathcal{M}_{g,n}$ \virg{trivially stabilizes} when $g \to \infty$, and the stable orbifold cohomology of $\mathcal{M}_{g,n}$ coincides with its ordinary stable cohomology. The only pairs $(g,n)$ for which $2g/3-2/3 > g-3+n$ occur when $(g,n) \in  \{(1,1), (2,0), (3,0), (4,0)\}$. In these special cases, the ranges for $k$ in Corollary \ref{corollario} are optimal, whereas in all other cases our ranges coincide with the ranges of stability for ordinary cohomology: $k < 2g/3-2/3$. The latter ranges are known to be optimal when $g\equiv 2 \pmod 3$. More details on the sharpness of the ranges for cohomological stability can be found in \cite[p.2]{wahl}. 

Combining \eqref{stabilization} with the isomorphisms \eqref{mw}, we see how the orbifold cohomology of $\mgn$ is explicitly computable in low degrees.

We now move to the proof of Theorem \ref{teoremaeta}. Thanks to Proposition \ref{codimensione} and to \eqref{etacodim}, what we have to prove is in fact
\begin{equation} \label{centrale}
|a(Y)- a(Y')| \leq \codim(Y) - (g-2+n)= 2g+2-3g'- \sum d_i,
\end{equation}
with equality only when $Y$ is the hyperelliptic twisted sector with $n$ marked Weierstrass points.%

We introduce some notation. Let $\Sigma$ be the set of proper divisors of $N$
\begin{displaymath}
\Sigma:= \{d \in \mathbb{N}| \ d \textrm{ divides } N, d \neq N\},
\end{displaymath} and let
\begin{align}  
a(Y)_{mark} &:=\frac{1}{N} \sum_{\gcd(i,N)=1} a_i \sum_{k=1}^{N-1} k \ \sigma(k,i),\\
a(Y)_{\sigma}&:=  \frac{1}{N} \sum_{\gcd(i,N)=\sigma} d_i \sum_{k=1}^{N-1} k \left( \left\{\frac{k i}{N}\right\} + \sigma(k,i) \right)\label{asigma}
\end{align}
We can rewrite formula \eqref{etamg} for the age of a twisted sector $Y$ as: \begin{displaymath}a(Y)=\frac{(3 g'-3)(N-1)}{2}+ a(Y)_{mark} +\sum_{\sigma \in \Sigma} a(Y)_{\sigma}.\end{displaymath} The term $a(Y)_{mark}$ is the contribution to the age of $Y$ coming from the marked points, and as such it is zero when $n=0$. Of course now we have the estimate
\begin{equation} \label{asigma1}
|a(Y)- a(Y')| \leq \left|a(Y)_{mark} - a(Y')_{mark}\right|+ \sum_{\sigma \in \Sigma} \left|a(Y)_{\sigma} - a(Y')_{\sigma}\right| .
\end{equation}
We can give estimates for each term in the right hand side of \eqref{asigma1}.
\begin{lemma}\label{generalized} The following inequalities hold:
\begin{align} \label{generalizedeq}
|a(Y)_{mark}-a(Y')_{mark}|& \leq \frac{N-2}{N} \sum_{\gcd(i,N)=1} a_i, \\
|a(Y)_{\sigma}-a(Y')_{\sigma}|& \leq \frac{(N-2 \sigma) (\frac{N}{\sigma}+5)}{6 N} \sum_{\gcd(i,N)= \sigma} d_i. \label{generalizedeq1}
\end{align}
\end{lemma}

\begin{proof} Let us begin with the contribution coming from the marked points. The left hand side of \eqref{generalizedeq} is equal to
\begin{equation} \label{equalto}
\frac{1}{N} \left| \sum a_i (\lambda(i)- \lambda(N-i)) \right|,
\end{equation}
where $\lambda(s)$ is the multiplicative inverse of $s$ modulo $N$. The maximum of the absolute value of \begin{displaymath}\lambda(i)- \lambda(N-i)= 2 \lambda(i) -N,\end{displaymath} is obtained when $i$ is either $1$ or $N-1$. 

As for the second inequality, we separate the two summands in the right hand side of \eqref{asigma}. For the first term, consider the function of $i$
\begin{displaymath}
g^{\sigma}_N(i):=\left|\sum_{k=1}^{N-1} k \left( \left\{ \frac{i k}{N}\right\}- \left\{ \frac{(N-i) k}{N}\right\}\right)\right|. 
\end{displaymath}
Its maximum among the values of $i$ such that $\gcd(i,N)=\sigma$ is obtained for $i=\sigma$ or for $i=N - \sigma$. From this, we obtain 
\begin{equation} \label{uno}
\left|\sum_{k=1}^{N-1} k  \left\{ \frac{i k}{N}\right\}- \left\{ \frac{(N-i) k}{N}\right\}\right| \leq g_N^{\sigma}(\sigma)= \frac{1}{6} \left(\frac{N}{\sigma}-1\right) (N- 2 \sigma).
\end{equation}
The second term is treated similarly to the contribution coming from the marked points. The maximum of the absolute value of 
\begin{displaymath} \label{due}
\sum_k( \sigma(k,i)- \sigma(k, N-i))=2i-N
\end{displaymath}
is obtained when $i$ is either $\sigma$ or $N- \sigma$. Combining this fact with \eqref{uno}, we get the desired inequality.
\end{proof}

\begin{proof} (of Theorem \ref{teoremaeta}) As we have already observed, it suffices to prove \eqref{centrale}.
By using the Riemann-Hurwitz formula \eqref{rh} to eliminate the variable $g$, the right hand side of \eqref{centrale} can be rearranged to:
\begin{displaymath}
(2N-3)g'-2 (N-2) + \sum_{\sigma \in \Sigma} (N- \sigma -1) \sum_{\gcd(i,N)= \sigma}  d_i.
\end{displaymath}
Now let us define for convenience the function:
\begin{equation}
\nonumber f_N(\sigma):=  (N- \sigma -1) - \frac{(N-2 \sigma) (\frac{N}{\sigma}+5)}{6 N}= \frac{\left( 6- \frac{1}{\sigma}\right) N^2 -  (6 \sigma + 9) N + 10 \sigma} {6 N}, 
\end{equation}
for any integer $N\geq 2$ and any $\sigma$ a real number between $1$ and $N/2$.
By using \eqref{generalizedeq} and \eqref{generalizedeq1}, in order to prove \eqref{centrale} it is enough to prove:
\begin{equation} \label{finalina}
%
- \frac{N-2}{N} \sum_{\gcd(i,N)=1} a_i+ \sum_{\sigma \in \Sigma} f_N(\sigma)  \sum_{\gcd(i,N)= \sigma}  d_i \geq (2-2g')(N-2)-g',
\end{equation}
with equality only in the case of the hyperelliptic twisted sector. Note that $a_i \leq d_i$ from inequality \eqref{markedpoints}. 

The left hand side of \eqref{finalina} is always nonnegative for any integer $N \geq 2$, because
\begin{displaymath}
\tilde{f_N}(1):= f_N(1)- \frac{N-2}{N}=\frac{(N-2)(5N-11)}{6N} \geq0.
\end{displaymath}
 Therefore when $g'>0$, the strict inequality \eqref{finalina} holds evidently, as the right hand side is strictly smaller than $0$. Thus all we have to prove is \eqref{finalina} when $g'=0$, 
a case in which we always have that $\sum d_i \geq 3$ (this follows from condition \eqref{rh} posing $g'=0$ and $g>0$).

We start with the case $g'=0$ and $\sum d_i \geq 4$. 
The function $f_N$ is concave (just look at its second derivative with respect to $\sigma$), thus it has its minimum either in $1$ or in $N/2$, 
and we have $f_N(N/2)= \frac{N-2}{2}$.
The following two inequalities hold in this case
\begin{align}
\tilde{f_N}(1) \sum d_i &\leq 2 (N-2), \label{first}\\
f_N(N/2) \sum d_i & \leq  2 (N-2), \label{second}
\end{align}
and they suffice to prove \eqref{finalina}. If \eqref{finalina} is an equality, then either \eqref{first} or \eqref{second} must be an equality. If $N$ equals $2$, we are precisely in the case of the hyperelliptic twisted sector. If $N>2$, the inequality \eqref{first} is strict, so \eqref{second} must be an equality and therefore $\sum d_i=4$. So if \eqref{finalina} is an equality, with $g'=0, N>2$ and $\sum d_i=4$, then $d_{\frac{N}{2}}=4$, but this implies $g=1$ by \eqref{rh}, hence $n \geq 1$, and this case does not exist because of \eqref{markedpoints}.
 
 So we are left with the case $g'=0$ and $\sum d_i=3$. A large number of twisted sectors still falls into this last category, but not the hyperelliptic twisted sector. We pose the three nonzero $d_i$'s to be $d_{\sigma_1}=d_{\sigma_2}=d_{\sigma_3}=1$. Then it suffices to prove the strict inequality
 
 \begin{equation} \label{finalina2}
 \left(6- \sum_{i=1}^3 \frac{1}{\sigma_i}\right)N^2- \left(3+ 6\sum_{i=1}^3 \sigma_i\right) N + 10 \sum_{i=1}^3 \sigma_i > 6 n (N-2).
 \end{equation}
  If $N$ is fixed, there are only finitely many possibilities for the variables involved in \eqref{finalina2}. The constraints are:
 \begin{equation} \begin{cases}\sigma_1 + \sigma_2 + \sigma_3 < N, \\ a \sigma_1 + b \sigma_2 + c \sigma_3 =N \quad \textrm{for some } a,b,c \in \mathbb{N}^+, \\ n \leq \left| \{i| \ \sigma_i=1\}\right| \leq 3, \\ \sigma_i \textrm{ divides } N, \sigma_i \neq N, \label{condizioni}
 \end{cases} \end{equation} where all the quantities involved are integers.
  The first is a consequence of Riemann-Hurwitz \eqref{rh} (assuming $g>1$), the second follows from \eqref{fan} and the third from \eqref{markedpoints}. From now on, we aim at proving \eqref{finalina2} for $N$ greater than a certain explicit constant. We will repeatedly use that the left hand side of \eqref{finalina2}, for fixed $n,N$, is a concave function in the domain of definition \eqref{condizioni}. We can also assume for convenience that $\sigma_1 \leq \sigma_2 \leq \sigma_3$.
 
 \begin{itemize} \item If $n=3$, then from \eqref{condizioni} we deduce $\sigma_1=\sigma_2=\sigma_3=1$. The inequality \eqref{finalina2} is satisfied when $N>11$.
 
 \item If $n=2$, from \eqref{condizioni}, we have that $\sigma_1=\sigma_2=1$. It is enough to check \eqref{finalina2} for the extreme values $\sigma_3=1$ and $\sigma_3=N/2$. The first follows from the case $n=3$, by checking the case of $\sigma_3=N/2$ we see that \eqref{finalina2} is valid when $N>22$.
 
 \item If $n=1$, from \eqref{condizioni} $\sigma_1=1$, so we have $1 \leq \sigma_2 \leq \sigma_3 \leq N/2$ and $\sigma_2+\sigma_3 < N-1$. It is enough to check the extremal values. The case when $\sigma_2=1$ follows from the case $n=2$. From the second point in \eqref{condizioni}, if $\sigma_3=N/2$, then $\sigma_2$ is either $1$ or $2$; in the latter case \eqref{finalina2} is valid when $N>14$. Finally, when $\sigma_2=\sigma_3=N/3$, \eqref{finalina2} is always valid.
 
 \item If $n=0$, we can assume $\sigma_i \geq 2$, since the other cases fall in the above paragraph. Moreover, there are six extremal cases that fulfill the first and the last of \eqref{condizioni}: \begin{displaymath} \hspace{1.1cm} \label{triple} (2,2,2), (2,2, \frac{N}{2}), (2,\frac{N}{3},\frac{N}{2}), (\frac{N}{7}, \frac{N}{3}, \frac{N}{2}), (\frac{N}{5}, \frac{N}{4}, \frac{N}{2}), (\frac{N}{4}, \frac{N}{3}, \frac{N}{3}). \end{displaymath} We check that \eqref{finalina2} for the extremal cases is satisfied when $N>36$ (the inequality is sharp in the case of the fourth triple).
 \end{itemize}
 
 To conclude the proof, we have to check that \eqref{centrale} holds in the cases when $g'=0$, $\sum d_i=3$ and $N<37$, which imply $g\leq 17$. These cases are only finitely many, and can be handled with the help of a computer program\footnote{The source code of a \texttt{C++} program that lists all twisted sectors of $\mathcal{M}_g$, each one with its age, is available at {\url{http://pcwww.liv.ac.uk/~pagani/twisted.cpp}}.}. \end{proof}
 
Let us conclude with some remarks.

\begin{remark} We list the number of twisted sectors with of $\mathcal{M}_g$ for $1 \leq g \leq 17$, to give an idea of its rapid growth:
\begin{displaymath}
(7,17,47,72,76,203,196,225,415,537,427,1040,811,779,1750,1860,1371).
\end{displaymath}
Then the number of twisted sectors of $\mathcal{M}_g$ with $g'=0$:
\begin{displaymath}
\hspace{-0.3cm}(7,16,43,65,64,193,163,207,372,485,359,983,657,866,1592,1636,1115).
\end{displaymath}
And, finally, the number of twisted sectors of $\mathcal{M}_g$ with $g'=0$ and $\sum d_i=3$:
\begin{displaymath}
\hspace{-1.1cm}(6,12,32,38,42,108,76,100,184,190,150,352,162,286,544,382,196).
\end{displaymath}
These final $2860$ twisted sectors are those for which we performed the computer assisted calculation mentioned in the last paragraph of the proof of Theorem \ref{teoremaeta}.
\end{remark}

\begin{remark} One can also ask what are the twisted sectors of small age in $\mathcal{M}_{g,n}$, after the marked hyperelliptic one. Here we list, for fixed $(g,n)$, the first twisted sectors in order of increasing age: marked hyperelliptic, marked bielliptic, \ldots,  (marked) double covers of curves of genus  $\lceil \frac{g}{2} \rceil$. We remark that the ranges of existence of those twisted sectors, in terms of $g$ and $n$, are, respectively
\begin{displaymath}
n \leq 2g+2, \ n \leq 2g-2, \ n \leq 2g-6,\ \ldots, \ n \leq 1+(-1)^g; 
\end{displaymath}  
their age is, respectively,
\begin{displaymath}
\frac{g-2+n}{2}, \ \frac{g-1+n}{2}, \ \frac{g+n}{2}, \ \ldots, \frac{3 \lfloor \frac{g}{2} \rfloor -1 - (-1)^g +n}{2}.
\end{displaymath}

After all these, there is one marked trigonal cyclic twisted sector (when $n \leq g+2$). Then the picture becomes more complicated, and we do not know the answer. For example, we have empirically observed that the minimum age among twisted sectors of codimension $k$ can be bigger than the minimum age among twisted sectors of codimension $k+1$.

The validity of the statements that we made in this remark require a long combinatorial proof along the lines of the proof of Theorem 1, which we do not include in this note as it is not really relevant to our scope.
\end{remark}

\begin{remark}
Condition \eqref{nonempty} has not been used in any of the steps of the proof of Theorem \ref{teoremaeta}, which could then have been stated slightly more generally for the twisted sectors of the moduli spaces of not necessarily connected  smooth curves of genus $g$.
\end{remark}

\begin{remark} There is no such thing as a stable cohomology in low degrees for $\overline{\mathcal{M}}_{g,n}$; it is a classical fact for example that even the second Betti number (which equals the dimension of the Picard group in this case) grows exponentially in $g$.  It still makes sense to ask for the twisted sector of minimum age of $\overline{\mathcal{M}}_{g,n}$, but the answer is much easier. To fix the ideas, we give the answer when $g+n>3$ and $g>0$: the cases in which the generalized hyperelliptic locus has codimension $>1$. Then the unique twisted sector of minimum age is the codimension-$1$ locus, consisting generically of a smooth elliptic curve glued at the origin to a smooth curve of genus $g-1$ carrying all the marked points, and with the automorphism induced by the pair (elliptic involution on the elliptic curve, identity). Its age is $\frac{1}{2}$: half its codimension.
\end{remark}


\begin{thebibliography}{2}

\bibitem [ALR] {alr} A. Adem, J. Leida, Y. Ruan, \emph{Orbifolds and stringy topology}, Cambridge Tracts in Mathematics, 171. Cambridge University Press, Cambridge, 2007. 


\bibitem [CR] {chenruan} W. Chen, Y. Ruan, \emph{A new cohomology theory of orbifold}, Comm. Math. Phys. 248 (2004), no. 1, 1--31.

\bibitem[DHVW1] {dhvw1}  L. Dixon, J. Harvey, C. Vafa, E. Witten,
  \emph{Strings On Orbifolds},
  Nuclear Physics  B (261), 1985.

\bibitem[DHVW2] {dhvw2}  L. Dixon, J. Harvey, C. Vafa, E. Witten,
  \emph{Strings On Orbifolds 2},
  Nuclear Physics  B (274), 1986.



\bibitem [Fa]{fantechi} B. Fantechi, \emph{Inertia stack of $\mathcal{M}_{g,n}$}, talk at the conference ``Modern Moduli Theory'', \ \texttt{MSRI}, February 2009.

\bibitem [Ha] {harer} J. Harer, \emph{Stability of the homology of the mapping class groups of orientable surfaces}, Ann. of Math. (2) 121 (1985), no. 2, 215--249.

\bibitem [JKK] {jkk} T. Jarvis, R. Kaufmann, T. Kimura \emph{ Stringy $K$-theory and the Chern character}, Invent. Math. 168 (2007), no. 1, 23--81.

\bibitem [Ka] {kawasaki} T. Kawasaki, \emph{The Riemann-Roch theorem for complex $V$-manifolds}, Osaka J. Math. 16 (1979), no. 1, 151-–159. 

\bibitem [Ki] {kirwan} F. Kirwan, \emph{Cohomology of moduli spaces}, Proceedings of the International Congress of Mathematicians, Vol. I (Beijing, 2002), 363--382, Higher Ed. Press, Beijing, 2002. 

\bibitem [Lo] {looijenga} E. Looijenga, \emph{Stable cohomology of the mapping class group with symplectic coefficients and of the universal Abel-Jacobi map}, J. Algebraic Geom. 5 (1996), no. 1, 135--150.

\bibitem [MW] {madsen} I. Madsen, M. Weiss, \emph{The stable moduli space of Riemann surfaces: Mumford's conjecture}, Ann. of Math. (2) 165 (2007), no. 3, 843--941.

\bibitem [P] {pagani2} N. Pagani, \emph{The Chen--Ruan cohomology of moduli of curves of genus $2$ with marked points}, Adv. Math. 229 (2012), no. 3, 1643--1687.




\bibitem [PT] {pt} N. Pagani, O. Tommasi, \emph{The orbifold cohomology of moduli of genus $3$ curves}, Manuscripta Math. 142 (2013), no. 3, 409--437, 

\bibitem [Wa] {wahl} N. Wahl, \emph{Homological stability for mapping class groups of surfaces}, Handbook of Moduli, Vol. III, 547--583. Advanced Lectures in Mathematics 26 (2012).

\end{thebibliography}
\end{document}